\documentclass[10pt,fleqn]{article}
\usepackage{anysize} 
\usepackage[utf8]{inputenc}
\usepackage[T1]{fontenc}

\usepackage[american]{babel}
\usepackage{amsmath,amsthm}
\usepackage{amssymb}
\usepackage{times}
\usepackage{url}
\usepackage{titlesec}
\usepackage{hyperref}
\usepackage{setspace}

\titleformat*{\section}{\normalsize\bfseries}

\newcommand{\dist}{\mathrm{dist}}
\newcommand{\R}{\mathbb{R}}
\newcommand{\Z}{\mathbb{Z}}
\newcommand{\Sym}{\mathfrak{S}}

\theoremstyle{theorem}
\newtheorem{theorem}{Theorem}
\newtheorem{lemma}[theorem]{Lemma}
\newtheorem{problem}[theorem]{Problem}

\begin{document}

\title{
	Barycenters of points that are constrained to a polytope skeleton
}

\author{\scshape\normalsize Pavle V. M. Blagojevi\'c\thanks{Inst.\ Mathematik, FU Berlin, Arnimallee 2, 14195 Berlin, Germany, \href{mailto:blagojevic@math.fu-berlin.de}{blagojevic@math.fu-berlin.de} and \href{mailto:ziegler@math.fu-berlin.de}{ziegler@math.fu-berlin.de}. This research was funded by the European Research Council via ERC grant agreement no.~247029 \emph{SDModels}.}
\and
\scshape\normalsize Florian Frick\setcounter{footnote}{6}\thanks{Inst.\ Mathematik, MA 8-1, TU Berlin, Str. des 17. Juni 136, 10623 Berlin, Germany, \href{mailto:frick@math.tu-berlin.de}{frick@math.tu-berlin.de}. Research funded by the German Science Foundation DFG via the Berlin Mathematical School.}
\and
\scshape\normalsize G\"unter M. Ziegler$^*$}

\date{17. November 2014}

\maketitle

\setstretch{1.24}

\begin{abstract}\noindent
	We give a short and simple proof of a recent result of Dobbins that any point in an $nd$-polytope is the barycenter of $n$ points in the $d$-skeleton. This new proof builds on the constraint method that we recently introduced to prove Tverberg-type results.
\end{abstract}

A problem originally motivated by mechanics is to determine whether each point in a polytope is the barycenter of points in an appropriate skeleton of the polytope \cite{barba2014weight}. Its resolution by Dobbins recently appeared in Inventiones.

\begin{theorem}[Dobbins \cite{dobbins2014point}]\label{thm:dobbins}
	For any $nd$-polytope $P$ and for any point $p \in P$, there are points $p_1, \dots , p_n$ in the $d$-skeleton of $P$ with barycenter $p = \tfrac{1}{n}(p_1 +\dots + p_n)$.
\end{theorem}

Here we simplify the proof of this theorem by using the idea of equalizing distances of points to a certain unavoidable skeleton via equivariant maps to force them into the skeleton. We recently used this setup to give simple proofs of old and new Tverberg-type results~\cite{blagojevic2014tverberg}. Thus we obtain the following slight generalization of Theorem~\ref{thm:dobbins}.

\begin{theorem}\label{thm}
	Let $P$ be a $d$-polytope, $p \in P$, and $k$ and $n$ positive integers with $kn \ge d$. Then there are points $p_1,\dots , p_n$ in the $k$-skeleton $P^{(k)}$ of $P$ with barycenter $p = \tfrac{1}{n}(p_1 +\dots + p_n)$.
\end{theorem}

More generally, one could ask for a characterization of all possibly inhomogeneous dimensions of skeleta and barycentric coordinates:

\begin{problem}
	For given positive integers $d$ and $n$ characterize the dimensions $d_1, \dots, d_n \ge 0$ and coefficients $\lambda_1, \dots, \lambda_n \ge 0$ with $\sum \lambda_i = 1$ such that for any $d$-polytope $P$ there are $n$ points $p_1 \in P^{(d_1)}, \dots, p_n \in P^{(d_n)}$ with $p = \lambda_1p_1 + \dots + \lambda_np_n$.
\end{problem}

Some remarks pertaining to this more general problem are contained in \cite{dobbins2014point}. Still of greater generality is the problem of characterizing the subsets $F \subseteq \R^d$ that contain a set of $n$ not necessarily distinct vectors $p_1, \dots, p_n \in F$ with $p_1 + \dots + p_n = 0$. In other words we are interested in those sets $F$ which contain a not necessarily embedded $(n-1)$-simplex with barycenter at the origin. This is equivalent to $F^n \subseteq \R^{d\times n}$ having non-empty intersection with $W_n^{\oplus d}$, where $W_n := \{(x_1, \dots, x_n) \in \R^n \: | \: \sum x_i = 0\}$. If $F$ is a set of dimension $k$ with $kn < d$, then generically this intersection is empty for dimension reasons. 

For the proof of Theorem \ref{thm} we will need the following lemma.

\begin{lemma}[Dold \cite{dold1983simple}, see also Matou\v{s}ek \cite{matousek:borsuk-ulam2}]\label{lem:dold}
	Let a non-trivial finite group $G$ act on an $n$-connected CW-complex $T$ and act linearly on an $(n+1)$-dimensional real vector space $V$. Suppose that the action of $G$ on $V\setminus\{0\}$ is free. Then any $G$-equivariant map $\Phi\colon T \to V$ has a zero.
\end{lemma}

Let from now on the symmetric group $\Sym_n$ act on the space of matrices $\R^{d\times n}$ by permuting columns.

\begin{theorem}\label{thm:constrained-barycenter}
	Let $n$ be prime, $d \ge 1$ be an integer, and $F \subseteq \R^d$ be closed. If there is an $(n-2)$-connected, $\Sym_n$-invariant subset $Q \subseteq W_n^{\oplus d} = \{(x_1, \dots, x_n) \in \R^{d\times n} \: | \: \sum x_i = 0\}$ such that for each $(x_1, \dots, x_n) \in Q$ there is an $i$ with $x_i \in F$, then there are $p_1, \dots, p_n \in F$ with $p_1 + \dots + p_n = 0$.
\end{theorem}

\begin{proof}
	The map $\Psi\colon Q \to \R^n, (x_1, \dots, x_n) \mapsto (\dist(x_1, F), \dots, \dist(x_n, F))$ is $\Sym_n$-equivariant. Denote the diagonal by $D := \{(y_1, \dots, y_n) \in \R^n \: | \: y_1 = \dots = y_n\}$. The map $\Psi$ induces by projection an $\Sym_n$-equivariant map $\Phi\colon~Q~\to~D^\perp = W_n$. The vector space $W_n$ is $(n-1)$-dimensional and $Q$ is $(n-2)$-connected, so $\Phi$ has a zero by Lemma \ref{lem:dold} applied to the subgroup $\Z/n$ of $\Sym_n$, which acts freely on $W_n \setminus \{0\}$. Let $(p_1, \dots, p_n) \in Q$ with $\Phi(p_1, \dots, p_n) = 0$. This is equivalent to $\dist(p_1, F) = \dots = \dist(p_n, F)$. There is an index $i$ such that $\dist(p_i,F) = 0$, or equivalently $p_i \in F$, since $F$ is closed. Thus all $p_j$ are in $F$. Since $Q \subseteq W_n^{\oplus d}$ we have $p_1 + \dots + p_n = 0$.
\end{proof}

This theorem can be extended to the case that $n$ is a prime power using a generalization of Dold's theorem to elementary Abelian groups; for a rather general version see \cite{blagojevic2012equivariant}.

For the proof of Theorem \ref{thm} we take for $Q$ a skeleton of a certain polytope. Before proceeding we point out that the following na\"ive approach using a configuration space/test map scheme to prove Theorem \ref{thm:dobbins} fails: the origin is the barycenter of $n$ points in the $d$-skeleton of $P$ if and only if the $\Sym_n$-equivariant map $(P^{(d)})^n~\to~\R^{dn}$, $(x_1, \dots, x_n) \mapsto \tfrac{1}{n}(x_1 + \dots + x_n)$ maps some point in $(P^{(d)})^n$ to $0$. However, for this map to be equivariant the $\Sym_n$-action on $\R^{dn}$ must be trivial, and thus $\Sym_n$-equivariant maps $(P^{(d)})^n \to \R^{dn}\setminus\{0\}$ exist. Dobbins's novel idea was to intersect with a test space in the domain to avoid this problem. We will employ that same idea. In contrast to Dobbins we need no requirement of general position. Thus we also do not need any kind of approximation in the proof.

\begin{proof}[Proof of Theorem \ref{thm}]
	We can assume that $p = 0$ is in the interior of $P$, otherwise we could restrict to a proper face of $P$ with the origin in its relative interior. Let first $n$ be prime. Consider again the linear space $W_n^{\oplus d} = \{(x_1, \dots, x_n) \in \R^{d\times n} \: | \: \sum x_i = 0\}$ of codimension $d$. Then $C := P^n \cap W_n^{\oplus d}$ is a polytope of dimension $(n-1)d$. The $(n-1)$-skeleton $C^{(n-1)}$ of $C$ is homotopy equivalent to a wedge of $(n-1)$-spheres and thus is $(n-2)$-connected.
	
	Let $(x_1, \dots, x_n) \in C^{(n-1)}$. By Theorem \ref{thm:constrained-barycenter} we need to show that one $x_i$ lies in $P^{(k)}$. Suppose for contradiction that $x_i \notin P^{(k)}$ for all $i = 1, \dots, n$. For each $x_i$ let $\sigma_i$ be the inclusion-minimal face of $P$ with $x_i \in \sigma_i$. We have that $\dim \sigma_i \ge k+1$. Each face of $C$ is of the form $(\tau_1 \times \dots \times \tau_n) \cap W_n^{\oplus d}$ with the $\tau_i$ faces of $P$. The point $(x_1, \dots, x_n)$ lies in the face $(\sigma_1 \times \dots \times \sigma_n) \cap W_n^{\oplus d}$ but in no proper subface. The dimension is $\dim \left((\sigma_1 \times \dots \times \sigma_n) \cap W_n^{\oplus d}\right) \ge n(k+1)-d\ge n$. Thus $(\sigma_1 \times \dots \times \sigma_n) \cap W_n^{\oplus d}\notin C^{(n-1)}$, which is a contradiction. The case for general $n$ follows by a simple iteration with respect to prime divisors, as in \cite{dobbins2014point}.
\end{proof}

\linespread{1.0}
\setlength{\parskip}{0cm}

\small
\providecommand{\noopsort}[1]{}
\providecommand{\bysame}{\leavevmode\hbox to3em{\hrulefill}\thinspace}
\providecommand{\MR}{\relax\ifhmode\unskip\space\fi MR }
\providecommand{\MRhref}[2]{%
  \href{http://www.ams.org/mathscinet-getitem?mr=#1}{#2}
}
\providecommand{\href}[2]{#2}

\end{document}